\documentclass[a4paper, 11pt]{amsart}   
\usepackage{mathptmx, amssymb,amscd,latexsym, eulervm}
\usepackage{changepage} 
\usepackage{amsmath}
\usepackage{amsthm}
\usepackage{mathdots}
\usepackage[pagebackref,colorlinks=true,linkcolor=blue,urlcolor=blue]{hyperref}
\usepackage{color}
\usepackage[onehalfspacing]{setspace}
\usepackage{tabularx}
\usepackage{amsfonts}
\usepackage{paralist}
\usepackage{aliascnt}
\usepackage[initials, lite]{amsrefs}
\usepackage{amscd}
\usepackage{blkarray}
\usepackage{mathbbol}
\usepackage{setspace}
\usepackage[inner=2.4cm,outer=2.4cm, bottom=3.2cm]{geometry}
\usepackage{tikz, tikz-cd}
\usepackage{calligra,mathrsfs}
\usepackage{comment}

\usepackage{tikz}
\usetikzlibrary{matrix}
\usetikzlibrary{arrows,calc}
\usetikzlibrary{decorations.markings,intersections,positioning,calc}

  \tikzset{mylabel/.style  args={at #1 #2  with #3}{
    postaction={decorate,
    decoration={
      markings,
      mark= at position #1
      with  \node [#2] {#3};
 } } } }
\allowdisplaybreaks

\BibSpec{collection.article}{%
	+{}  {\PrintAuthors}                {author}
	+{,} { \textit}                     {title}
	+{.} { }                            {part}
	+{:} { \textit}                     {subtitle}
	+{,} { \PrintContributions}         {contribution}
	+{,} { \PrintConference}            {conference}
	+{}  {\PrintBook}                   {book}
	+{,} { }                            {booktitle}
	+{,} { }                            {series}
	+{, vol.} { }                            {volume}
	+{,} { }                            {publisher}
	+{,} { \PrintDateB}                 {date}
	+{,} { pp.~}                        {pages}
	+{,} { }                            {status}
	+{,} { \PrintDOI}                   {doi}
	+{,} { available at \eprint}        {eprint}
	+{}  { \parenthesize}               {language}
	+{}  { \PrintTranslation}           {translation}
	+{;} { \PrintReprint}               {reprint}
	+{.} { }                            {note}
	+{.} {}                             {transition}
	+{}  {\SentenceSpace \PrintReviews} {review}
}
\AtBeginDocument{%
	\def\MR#1{}
}

\makeatletter
\@namedef{subjclassname@2020}{%
	\textup{2020} Mathematics Subject Classification}
\makeatother

\newcommand{\lk}{\mathrm{lk}}
\newcommand{\st}{\mathrm{st}}
\newcommand{\dist}{\mathrm{dist}}
\newcommand{\cost}{\mathrm{cost}}

\newcommand{\tr}{\textrm{tr}}

\def\opn#1#2{\def#1{\operatorname{#2}}}
\opn\Cl{Cl} \opn\deg{deg} \opn\Stab{Stab} \opn\aff{aff} \opn\div{div}
\opn\cone{cone} \opn\End{End} \opn\mod{mod}  \opn\pdim{pdim} \opn\diag{diag} \opn\vert{vert} \opn\m{m} \opn\V{V}
\opn\Cone{Cone} \opn\Pyr{Pyr} \opn\max{max} \opn\min{min} \opn\int{int} \opn\rev{rev} \opn\ker{ker} \opn\lat{lat} \opn\pull{pull}
\opn\cok{coker} \opn\ant{ant}
\opn\inte{int}

\newcommand{\kk}{\mathbb{k}}

\newcommand{\NN}{\normalfont\mathbb{N}}
\newcommand{\ZZ}{\normalfont\mathbb{Z}}

\newcommand{\MM}{{\normalfont\mathfrak{M}}}

\newcommand{\pp}{{\normalfont\mathfrak{p}}}
\newcommand{\aaa}{\mathbf{a}}
\newcommand{\bbb}{\mathbf{b}}
\newcommand{\eee}{\mathbf{e}}

\newcommand{\ee}{{\normalfont\mathbf{e}}}
\newcommand{\Hom}{\normalfont\text{Hom}}

\newcommand{\FF}{\normalfont\mathcal{F}}

\newcommand{\Spec}{\normalfont\text{Spec}}

\def\f0{\mathbf{0}}

\def\1{\mathbf{1}}

\newtheorem{theorem}{Theorem}[section]

\newtheorem{headthm}{Theorem}

\newaliascnt{headcor}{headthm}

\aliascntresetthe{headcor}

\newaliascnt{headconj}{headthm}

\aliascntresetthe{headconj}

\newaliascnt{corollary}{theorem}
\newtheorem{corollary}[corollary]{Corollary}
\aliascntresetthe{corollary}

\newaliascnt{claim}{theorem}

\aliascntresetthe{claim}

\newaliascnt{lemma}{theorem}
\newtheorem{lemma}[lemma]{Lemma}
\aliascntresetthe{lemma}

\newaliascnt{conjecture}{theorem}

\aliascntresetthe{conjecture}

\newaliascnt{proposition}{theorem}
\newtheorem{proposition}[proposition]{Proposition}
\aliascntresetthe{proposition}

\theoremstyle{definition}
\newaliascnt{definition}{theorem}
\newtheorem{definition}[definition]{Definition}
\aliascntresetthe{definition}

\newaliascnt{notation}{theorem}

\aliascntresetthe{notation}

\newaliascnt{example}{theorem}

\aliascntresetthe{example}

\newaliascnt{examples}{theorem}

\aliascntresetthe{examples}

\newaliascnt{remark}{theorem}
\newtheorem{remark}[remark]{Remark}
\aliascntresetthe{remark}

\newaliascnt{question}{theorem}

\aliascntresetthe{question}

\newaliascnt{questions}{theorem}

\aliascntresetthe{questions}

\newaliascnt{problem}{theorem}

\aliascntresetthe{problem}

\newaliascnt{construction}{theorem}

\aliascntresetthe{construction}

\newaliascnt{setup}{theorem}

\aliascntresetthe{setup}

\newaliascnt{algorithm}{theorem}

\aliascntresetthe{algorithm}

\newaliascnt{observation}{theorem}

\aliascntresetthe{observation}

\newaliascnt{defprop}{theorem}

\aliascntresetthe{defprop}

\def\equationautorefname~#1\null{(#1)\null}
\def\sectionautorefname~#1\null{Section #1\null}
\def\subsectionautorefname~#1\null{\S #1\null}

\title{The canonical trace of Stanley--Reisner rings \\that are Gorenstein on the punctured spectrum}

\author[S. Miyashita]{Sora Miyashita}
\address[S. Miyashita]{Department of Pure And Applied Mathematics, Graduate School Of Information Science And Technology, Osaka University, Suita, Osaka 565-0871, Japan}
\email{u804642k@ecs.osaka-u.ac.jp}

\author[M. Varbaro]{Matteo Varbaro}
\address[M. Varbaro]{Dipartimento di Matematica, Universit\'a di Genova, Italy}
\email[M. Varbaro]{varbaro@dima.unige.it}

\dedicatory{We dedicate this work with our deepest gratitude to the memory of Professor J\"urgen Herzog (1941-2024)}

\date{\today}

\begin{document}

	\maketitle

\begin{abstract}
In this paper we prove that nearly Gorenstein Stanley--Reisner rings of dimension at least 3 are indeed Gorenstein. By previous work of the first author this yields a complete characterization of nearly Gorenstein Stanley--Reisner rings. 
We also show that a Cohen--Macaulay Stanley--Reisner ring is Gorenstein on the punctured spectrum if and only if either it is nearly Gorenstein or its canonical trace is the square of its irrelevant maximal ideal, and that the latter case happens exactly for non--orientable homology manifolds. 
 \end{abstract}

	\section{Introduction}

The trace of the canonical module $\tr(\omega_R)$ of a Cohen--Macaulay local (or graded) ring $R$, describes the non-Gorenstein locus of the corresponding algebraic variety~(\cite[Lemma 6.19]{herzog1971canonical} or \cite[Lemma 2.1]{herzog2019trace}), sparking active research on the ``canonical trace'' in recent years~(\cite{dao2020trace,hibi2021nearly,herzog2019trace,herzog2019measuring,celikbas2023traces,ficarra2024canonical,ficarra2024canonical!}). Of particular interest is the class of Cohen--Macaulay rings $R$ whose canonical trace contains the irrelevant maximal ideal $\MM_R$, a subclass of rings that are Gorenstein on the punctured spectrum. These rings, reintroduced as ``nearly Gorenstein rings'' by Herzog, Hibi and Stamate in \cite{herzog2019trace}, have since become a focal point of study, with numerous works exploring their properties across various contexts (see~\cite{huneke2006rings,ding1993note,miyashita2024linear,hall2023nearly,miyazaki2021Gorenstein,caminata2021nearly,kumashiro2023nearly,bagherpoor2023trace,srivastava2023nearly,hibi2021nearly,lu2024chain,lyle2024annihilators,miyazaki2024non,jafari2024nearly}).

In particular, nearly Gorenstein Stanley--Reisner rings arising from simplicial complexes of low dimensions were investigated in \cite[Section~4]{miyashita2024levelness}.
In conclusion, the study by \cite{miyashita2024levelness} prompts the following question:
\emph{``Is nearly Gorensteinness equivalent to Gorensteinness for Cohen–Macaulay Stanley–Reisner rings of dimension three or higher?''}.
In this paper, we classify the canonical trace of Stanley--Reisner rings that are Gorenstein on the punctured spectrum.
In doing so, we affirmatively resolve the aforementioned question.

 \begin{headthm}\label{main1}
 Let $\Delta$ be a simplicial complex, $R=\kk[\Delta]$ be the Stanley--Reisner ring of $\Delta$ over a field $\kk$, and assume that $R$ is Cohen--Macaulay. The following statements hold:
 \begin{itemize}
 \item[(X)]
$R$ is Gorenstein on the punctured spectrum
if and only if 
$\tr(\omega_R) = \MM_R^i$ for some $i \in \{0,1,2\}$.
 \item[(Y)] The following conditions are equivalent:
 \begin{itemize}
 \item[(1)] $\tr(\omega_R)=\MM_R$, that is, $R$ is nearly Gorenstein but not Gorenstein;
 \item[(2)] $\Delta$ is isomorphic either to a disjoint union of $n \geq 3$ vertices or to a path of length $n \geq 3$.
 \end{itemize}
 \item[(Z)] The following conditions are equivalent:
 \begin{itemize}
 \item[(1)] $\tr(\omega_{R})=\MM_{R}^2$,
 that is, $R$ is Gorenstein on the punctured spectrum but not nearly Gorenstein;
 \item[(2)] $\Delta$ is a non-$\kk$-orientable $\kk$-homology manifold.
 \end{itemize}
\end{itemize}
\end{headthm}


Note that the above theorem implies that, if $\kk$ has characteristic 2, a Cohen--Macaulay Stanley--Reisner ring is Gorenstein on the punctured spectrum if and only if it is nearly Gorenstein. Furthermore, in any characteristic, if $\dim(\Delta)\geq 2$ and $\kk[\Delta]$ is Gorenstein on the punctured spectrum, then either it is Gorenstein or $\Delta$ is a $\kk$-homology manifold (for this statement the Cohen--Macaulay assumption can be replaced by the weaker condition that $\Delta$ is connected, see Corollary \ref{c:puncGisHomologyMani}).

In the proof of our main result, we elucidate the relationship between level rings (\cite{stanley2007combinatorics})---one of the earliest generalizations of Gorenstein rings in the context of graded rings---and rings that are Gorenstein on the punctured spectrum.
The interplay of the class of almost Gorenstein rings~(\cite{goto2015almost})
 with the one of nearly Gorenstein rings has been explored in works such as \cite{miyashita2024levelness,miyashita2024comparing,moscariello2021nearly}. For instance, in the case of standard graded affine semigroup rings, it has been shown that the intersection of nearly Gorenstein and almost Gorenstein rings includes only a small subset of non-Gorenstein examples~(see~\cite[Theorem 6.1]{miyashita2024comparing}). Building on these insights, we establish as a consequence of this study that a similar phenomenon occurs for Stanley--Reisner rings, further enriching our understanding of these generalizations.

\begin{headthm}\label{main2}
The following statements hold:
\begin{itemize}
\item[(X)] Every Stanley--Reisner ring that is Gorenstein on the punctured spectrum is level.
\item[(Y)] Let $\Delta$ be a simplicial complex
and assume that $R=\kk[\Delta]$ is Cohen--Macaulay and not Gorenstein.
Then the following conditions are equivalent:
\begin{itemize}
\item[(1)] $R$ is Gorenstein on the punctured spectrum and almost Gorenstein;
\item[(2)] $R$ is nearly Gorenstein and almost Gorenstein;
\item[(3)] $R$ is nearly Gorenstein;
\item[(4)] $\Delta$ is isomorphic either to a disjoint union of $n \geq 3$ vertices or to a path of length $n \geq 3$.
\end{itemize}
\end{itemize}
\end{headthm}

\subsection*{Outline}
The outline of this paper is as follows: Proposition \ref{prop:main1} and Propositions \ref{p:nearlyG=G}, \ref{prop:main2} play a fundamental role in the proof of our main theorem. In Section \ref{sect2}, we introduce several key definitions of simplicial complexes that are central to our arguments and prove Proposition \ref{prop:main1}. Then, in Section \ref{sect3}, we utilize Gr\"abe's description of canonical modules of Stanley--Reisner rings to establish Propositions \ref{p:nearlyG=G}, \ref{prop:main2}, thereby completing the proofs of the described results.

\section{Stanley--Reisner rings which are Gorenstein on the punctured spectrum}
\label{sect2}

The purpose of this section is to lay the groundwork for the discussions of our main results and to establish Proposition \ref{prop:main1}. Throughout this paper,
we denote the set of nonnegative integers by $\NN$ and the set of integers by $\ZZ$.
Moreover, for a Noetherian positively graded ring $R=\bigoplus_{i \ge 0} R_i$,
we always assume that $R_0$ is a field.
In this context, notice that every Noetherian positively graded ring $R$
has a unique graded maximal ideal
${\MM_R}$ and admits a graded canonical module $\omega_R$.
Let $a_R$ denote the $a$-invariant of $R$.
Notice that $a_R = - \min \{j : (\omega_R)_j \neq 0 \}$.
For a $\ZZ$-graded $R$-module $M$, let $\tr_R(M)$ denote the sum of ideals $\phi(M)$
with $\phi \in \Hom^*_R(M,R):=\bigoplus_{i \in \ZZ}\Hom^i_R(M,R)$ (note that $\Hom^*_R(M,R)=\Hom_R(M,R)$ if $M$ is finitely generated). That is,
\[\tr_R(M)=\sum_{\phi \in \Hom^*_R(M,R)}\phi(M)\]
where $\Hom^i_R(M,R)=\{ \phi \in \Hom_R(M,R): \phi(M_k) \subseteq R_{k+i} \text{\;for any $k \in \ZZ$} \}.$
When there is no risk of confusion about the ring we simply write $\tr(M)$.
It is known that \(\tr(\omega_R)\) describes the non-Gorenstein locus if \(R\) is Cohen--Macaulay~(see \cite[Lemma 6.19]{herzog1971canonical} or \cite[Lemma 2.1]{herzog2019trace}).
On the other hand, even if \(R\) is a Noetherian ring such that $\dim(R/\pp)=\dim(R)$ for every minimal prime ideal $\pp$ of $R$, it is known that \(\tr(\omega_R)\) describes the non-quasi-Gorenstein locus of $R$~(\cite[Corollary 3.4]{aoyama1985endomorphism}). Building on these observations, this paper also frequently addresses the canonical trace of non-Cohen--Macaulay Noetherian rings.
We say that $R$ is {\it Gorenstein on the punctured spectrum} if $R_\pp$ is Gorenstein for every graded prime ideal $\pp$ of $R$ with $\pp \neq \MM_R$.
Note that this definition does not require
$R$ itself to be Cohen--Macaulay. In the following three definitions we assume that $R$ is Cohen--Macaulay.
\begin{definition}
We say that $R$ is $\textit{nearly Gorenstein}$~{\cite{herzog2019trace}} if $\tr(\omega_R) \supseteq {\MM_R}$.
It is known that $R$ is Gorenstein on the punctured spectrum $\Spec(R) \setminus \{ \MM_R \}$
if and only if $\sqrt{\tr(\omega_R)} \supseteq \MM_R$~(see \cite[Lemma 2.1]{herzog2019trace}).
In particular, every nearly Gorenstein ring is Gorenstein on the punctured spectrum.
\end{definition}

Below, we introduce the definitions of level rings and almost Gorenstein (graded) rings, which are other generalizations of Gorenstein graded rings.

\begin{definition}
We say that $R$ is {\it level}~(\cite{stanley2007combinatorics}) if $\omega_R$ is generated in a single degree. Equivalently the socle of an Artinian reduction of $R$ is concentrated in a single degree.

We say that $R$ is {\it almost Gorenstein}~(\cite[Section 10]{goto2015almost}) if there exists an $R$-monomorphism $\phi: R \hookrightarrow \omega_R(-{a_R})$ of degree 0 such that $C:=\cok(\phi)$ is either the zero module or an Ulrich $R$-module,
i.e. $C$ is Cohen--Macaulay (which in this case is automatic) and $\mu(C) = e(C)$.
Here, $\mu(C)$ (resp. $e(C)$) denotes the number of elements in a minimal generating system of $C$ (resp. the multiplicity of $C$ with respect to $\MM_R$) and
$\omega_R(-{a_R})$ is the graded $R$-module having the same underlying $R$-module structure as $\omega_R$,
where $[\omega_R(-{a_R})]_n = [\omega_R]_{n-{a_R}}$ for all $n \in \ZZ$.
\end{definition}

The following notion plays an important role in the proof of Proposition \ref{prop:main2}.

\begin{definition}[{\cite{gasanova2022rings}}]
Suppose that $R$ is not Gorenstein.
$R$ is said to be of \textit{Teter type}
if there exists a graded $R$-homomorphism $\phi: \omega_R \to R$ such that  $\operatorname{tr}_R(\omega_R) = \phi(\omega_R)$.
It is known that if $R$ is generically Gorenstein (e.g. reduced), then $R$ is not of Teter type (see \cite[Theorem 2.1]{gasanova2022rings}).
\end{definition}


Next, we summarize the terminology related to Stanley--Reisner rings that is necessary for this paper.
Let $S=\kk[x_1,\ldots ,x_n]$ be a polynomial ring in $n$ variables over a field $\kk$ equipped with the standard grading $\deg(x_i)=1$ for any $i=1,\ldots ,n$.

Recall that, given a simplicial complex $\Delta$ on $n$ vertices, the {\it Stanley--Reisner ideal} $I_{\Delta}\subseteq S$ is the ideal generated by the monomials $x_{i_1}x_{i_2}\cdots x_{i_s}$ such that $\{i_1,\ldots ,i_s\}\notin\Delta$.
This association yields a bijection
between simplicial complexes on $n$ vertices and squarefree monomial ideals of $R$. The quotient 
$\kk[\Delta]=S/I_{\Delta}$ is called the {\it Stanley--Reisner} ring of $\Delta$. Let us remind that the dimension of $\sigma\in\Delta$ is $\dim\sigma=|\sigma|-1$ and the dimension of $\Delta$ is $\dim\Delta=\max\{\dim\sigma:\sigma\in\Delta\}$. We say that $\sigma\in\Delta$ is an {\it $i$-face} of $\Delta$ if $\dim\sigma=i$.
A face of $\Delta$ maximal by inclusion is called a {\it facet}. We write $\FF(\Delta)$ for the set of facets of $\Delta$, and say that $\Delta$ is {\it pure} if $\dim\sigma=\dim \Delta$ for any $\sigma\in\FF(\Delta)$. A simplicial complex $\Delta$ is {\it strongly connected} if, for any $\sigma,\tau\in\FF(\Delta)$ there exist $\sigma_0,\sigma_1,\sigma_2,\ldots ,\sigma_k\in\FF(\Delta)$ such that $\sigma_0=\sigma$, $\sigma_k=\tau$ and $\dim \sigma_i\cap \sigma_{i-1}=\dim\Delta-1$ for all $i=1,\ldots ,k$. Notice that a strongly connected simplicial complex is necessarily pure.
Given a face $\sigma\in\Delta$, its {\it star, costar} and {\it link} are the following simplicial complexes: 
\begin{align*}
\st_{\Delta}\sigma=\{\tau\in\Delta:\tau\cup\sigma\in\Delta\}, \\
\cost_{\Delta}\sigma=\{\tau\in\Delta:\tau\not\supseteq\sigma\}, \\
\lk_{\Delta}\sigma=\{\tau\in\Delta:\tau\cup\sigma\in\Delta \mbox{ and }\tau\cap\sigma=\emptyset\}. 
\end{align*}
Notice that $\lk_{\Delta}\sigma=\st_{\Delta}\sigma\cap \cost_{\Delta}\sigma$. A face $\tau\in\Delta$ is a {\it cone} if $\tau\in\sigma$ for any $\sigma\in\FF(\Delta)$.

A simplicial complex $\Delta$ is called {\it normal} if $\lk_{\Delta}\sigma$ is connected for any $i$-face with $i\leq \dim\Delta-2$. Since $\dim\emptyset =-1$ and $\lk_{\Delta}\emptyset =\Delta$, a normal simplicial complex of dimension at least 1 is connected. It turns out that, if $\Delta$ is normal, then $\lk_{\Delta}\sigma$ (equivalently $\st_{\Delta}\sigma$) is strongly connected for all $\sigma\in \Delta$~(see \cite[Proposition 11.7]{Bj}). In particular a normal simplicial complex is pure. 
Indeed, it is known that $\Delta$ is normal if and only if $\kk[\Delta]$ satisfies Serre's condition $(S_2)$.
We say that a $d$-dimensional simplicial complex $\Delta$ is:
\begin{itemize}
\item a {\it $\kk$-homology sphere} if
for any $\sigma \in \Delta$, we have
\[
\widetilde{H}_{i}{(\lk_\Delta(\sigma), \kk)} = 
\begin{cases}
    \kk, & \text{if } i = \dim (\lk_\Delta(\sigma)), \\
    0, & \text{otherwise}.
\end{cases} \qquad - (*)
\]
\item a {\it $\kk$-homology manifold}
(without boundary)
if $\Delta$ is connected and $(*)$ holds for any $\sigma \in \Delta$ such that $\sigma \neq \emptyset$.
\item a {\it pseudomanifold} if it is strongly connected and any $(d-1)$-face is contained in exactly two facets (which means that the equation (*) below holds when $\sigma$ is a facet of $\Delta$).
\end{itemize}
Given a simplicial complex $\Delta$, by a classical theorem of Hochster $\kk[\Delta]$ is Gorentein if and only if $\lk_\Delta(\sigma)$ is a $\kk$-homology sphere, where $\sigma$ is the cone face of $\Delta$ maximal by inclusion (e.g. see Theorem 5.6.1 of \cite{bruns1998cohen}).
As a last piece of notation, we say that a $d$-dimensional pseudomanifold $\Delta$ is {\it orientable} if $H_d(\Delta;\ZZ)\neq 0$, and that is {\it $\kk$-orientable} if $H_d(\Delta;\kk)\neq 0$.
\begin{remark}
Note that the link of a pseudomanifold might fail to be a pseudomanifold (since it might be not strongly connected, think at a pinched torus), for this reason sometimes it is more natural to consider normal pseudomanifolds.
Furthermore notice that $\Delta$ is a $\kk$-homology manifold if and only if $\lk_\Delta v$ is a homology sphere for any vertex $v$ of $\Delta$.
\end{remark}

The following is well-known, we include a proof for the convenience of the reader:

\begin{lemma}\label{l:orientability}
Let $\Delta$ be a $d$-dimensional pseudomanifold. Then:
\begin{enumerate}
\item $\Delta$ is orientable if and only if $\widetilde{H}_d(\Delta;\ZZ)\cong \ZZ$.
\item $\Delta$ is orientable if and only if $\widetilde{H}_d(\Delta;\kk)\cong \kk$.
\item The following are equivalent:
\begin{itemize}
\item[(a)] $\Delta$ is orientable.
\item[(b)] $\Delta$ is $\kk$-orientable for any field $\kk$.
\item[(c)] $\Delta$ is $\kk$-orientable for some field $\kk$ of characteristic different from 2
\end{itemize}
\item If $\kk$ has characteristic 2, $\Delta$ is $\kk$-orientable.
\end{enumerate}
\end{lemma}
\begin{proof}
If $d\leq 0$ the statement is obvious, so we assume $d\geq 1$, in which case, for any unitary commutative ring $R$, $\widetilde{H}_d(\Delta;R)\cong H_d(\Delta;R)$. For all $i=-1,\ldots ,d$, let $F_i(R)$ be the free $R$-module with basis $1_{\sigma}$ where $\sigma$ is an $i$-face of $\Delta$. Then $H_d(\Delta;R)$ is the kernel of the map of $R$-modules $\partial:F_d(R)\rightarrow F_{d-1}(R)$  defined by $\partial(1_{\sigma})=\sum_{k=0}^d(-1)^k1_{\sigma\setminus\{a_k\}}$ where $\sigma=\{a_0<\ldots <a_d\}\in\Delta$.
Consider a nonzero cycle
\[v=\sum_{\alpha\in\FF(\Delta)}a_{\alpha}1_{\alpha}\in \ker(\partial)\subseteq F_d(R),\] 
where $a_{\alpha}\in R$. Since $v\neq 0$, there exists $\alpha\in \FF(\Delta)$ such that $a:=a_{\alpha}\neq 0$. We want to show that there exists a function $s:\FF(\Delta)\to \{+1,-1\}$ such that $a_{\beta}=s(\beta)a$ for all $\beta\in \FF(\Delta)$. Once this is proved, we have that $v=aw$ where $w=\sum_{\beta\in\FF(\Delta)}s(\beta)1_{\beta}\in \ker(\partial)\subseteq F_d(R)$, so $\ker(\partial)$ is the cyclic $R$-module $wR\subset F_d(R)$.

Put $X_1=\{\alpha\}$ and assume to having constructed a set $X_r$ consisting of $r$ facets of $\Delta$ such that the function $s$ is defined on $X_r$ and $a_{\gamma}=s(\gamma)a$ for all $\gamma\in X_r$. If $X_r=\FF(\Delta)$ we are done. Otherwise, let $\beta=\{j_0,\ldots ,j_d\}\in\FF(\Delta)\setminus X_r$ be such that there is $\gamma=\{i_0,\ldots ,i_d\}\in X_r$ with $\gamma\cap\beta=\gamma\setminus\{i_k\}=\beta\setminus\{j_h\}=:\xi$ (such a $\beta$ exists because $\Delta$ is strongly connected). Since $\gamma$ and $\beta$ are the only two facets of $\Delta$ containing $\xi$, if $\partial(v)=\sum_{\mu}b_{\mu}1_{\mu}$ (where the summation runs over the $(d-1)$-faces $\mu\in\Delta$), then $b_{\xi}=(-1)^k s(\gamma)a+(-1)^h a_{\beta}$. Since $\partial(v)= 0$, $b_{\xi}$ must be 0, i.e. $a_{\beta}=(-1)^{k+h+1}s(\gamma)a$. So we put $s(\beta)=(-1)^{k+h+1}s(\gamma)$ and set $X_{r+1}=X_r\cup \{\beta\}$. Since $\Delta$ is strongly connected, continuing this way we conclude.

It follows that, for any unitary commutative ring $R$, $\widetilde{H}_d(\Delta;R)\neq 0$ if and only if $\widetilde{H}_d(\Delta;R)\cong R$. So (1) and (2) follow. As for (4), it is obvious that $\sum_{\alpha\in\FF(\Delta)}1_{\alpha}$ is a nonzero element of $\ker(\partial)=\widetilde{H}_d(\Delta;\kk)$ if $\kk$ has characteristic 2.

For (3), by the universal coefficient theorem $H_d(\Delta;\ZZ)\otimes_{\ZZ}\kk$ injects into $H_d(\Delta;\kk)$, so $(a)\Rightarrow (b)$ follows. The implication $(b)\Rightarrow (c)$ is obvious. For $(c)\Rightarrow (a)$, consider $w=\sum_{\beta\in\FF(\Delta)}s(\beta)1_{\beta}$ in $\ker(\partial)=H_d(\Delta;\kk)\subseteq F_d(\kk)$ described above. We can consider $w$ as an element of $F_d(\ZZ)$, and doing so we have that $\partial(w)=\sum_{\mu}b_{\mu}1_{\mu}$ (where the summation runs over the $(d-1)$-faces $\mu\in\Delta$) where $b_{\mu}\in\{-2,0,+2\}$ for all $\mu$ since $\Delta$ is a pseudomanifold. Because $\partial(w)=0$ in $F_d(\kk)=F_d(\ZZ)\otimes_{\ZZ}\kk$ and $\pm 2 \neq 0$ in $\kk$, we must have $b_{\mu}=0 \ \forall \ \mu$, so $\partial(w)=0$ in $F_{d-1}(\ZZ)$ and $0\neq w\in H_d(\Delta;\ZZ)$.
\end{proof}


\begin{lemma}\label{lem:pseudoimplieshom}
Let $\Delta$ be a $d$-dimensional connected simplicial complex and assume that $\kk[\Delta]$ is Gorenstein on the punctured spectrum.
If $\Delta$ is a pseudomanifold, then $\Delta$ is a $\kk$-homology manifold.
\end{lemma}
\begin{proof}
We will show the contrapositive.
Suppose that $\Delta$ is not a $\kk$-homology manifold. That is, there exists $\sigma \in \Delta$ such that $\sigma \neq \emptyset$ and $\lk_\Delta(\sigma)$ is not a $\kk$-homology sphere. Since $\kk[\lk_\Delta(\sigma)]$ is Gorenstein, there exists a cone point $v$ of $\lk_\Delta(\sigma)$.
We fix $\tau \in \FF(\lk_\Delta(\sigma))$ and define $\pi := \tau \cup \sigma \in \FF(\Delta)$ (note that $v\in\tau\subset \pi$). We then define $\tau' := \pi \setminus \{v\}$.
Note that $\dim(\tau') = d-1$. To show that $\Delta$ is not a pseudomanifold, it suffices to prove that there is exactly one facet of $\Delta$ containing $\tau'$, namely $\pi$.
To this goal, take any facet $\alpha \in \FF(\Delta)$ such that $\alpha \supset \tau'$. Since $v \notin \sigma$, then $\tau'\supset \sigma$. It follows that $\alpha \supset \sigma$.
Consequently $\alpha \setminus \sigma \in \FF(\lk_\Delta(\sigma))$, so $v \in \alpha$ and thus $\alpha \supseteq \tau' \cup \{v\} = \pi$.
Because $\pi \in \FF(\Delta)$, this implies $\alpha = \pi$.
\end{proof}

\begin{remark}\label{Nat}
In general, a pseudomanifold, even if it is normal, is not necessarily a $\kk$-homology manifold.
In fact, there exists a three-dimensional example $\Delta$ as follows:
Let $T$ be the minimal triangulation of the 2-dimensional torus with 7 vertices.
Define $\Delta$ as a simplicial complex on the vertex set $[9]$ by $\Delta = (T * \{8\}) \cup (T * \{9\})$.
Then it follows that $\Delta$ is a normal pseudomanifold but not a $\kk$-homology manifold because $\lk_\Delta(8)=T$ is not a $\kk$-homology sphere.
\end{remark}

The following is the first proposition that plays a crucial role in proving the main result.

\begin{proposition}\label{prop:main1}
Let $\Delta$ be a $d$-dimensional connected simplicial complex.
If $\kk[\Delta]$ is Gorenstein on the punctured spectrum, then one of the following must hold:
\begin{itemize}
\item[(a)] $d \leq 1$;
\item[(b)] There exists a cone point of $\Delta$;
\item[(c)] $\Delta$ is a $\kk$-homology manifold.
\end{itemize}
\end{proposition}
\begin{proof}
First of all, notice that, since $\Delta$ is connected and Gorenstein on the punctured spectrum, it is normal, and hence strongly connected. Assume that $d \ge 2$ and there is no cone point of $\Delta$. 
We show (c). By Lemma \ref{lem:pseudoimplieshom} it is enough to show that $\Delta$ is a pseudomanifold. 
Suppose the contrary. Then there exists a face $\tau \in \Delta$ such that $\dim(\tau)=d-1$ and $\lk_\Delta(\tau)$ does not consist of two distinct vertices. Since $\lk_\Delta(\tau)$ is Gorenstein, it must be a single vertex $x$ of $\Delta$.
Put $\pi=\tau \cup \{ x \}$, $X:=\{ \sigma \in \FF(\Delta) : x \notin \sigma\}$
and
\[
\delta = \dist(\tau, X) := \min \left\{
\begin{array}{l}
k \in \mathbb{N} : \text{there exists a sequence } \pi=\sigma_0, \sigma_1, \cdots, \sigma_k \text{ such that } \\
\sigma_k \in X, \; \sigma_i \in \mathcal{F}_d(\Delta)  \text{\; and \;} |\sigma_i \cap \sigma_{i+1}| = d - 1 \text{ for any } i = 0, \dots, k-1
\end{array}
\right\}.
\]

Observe that $X \neq \emptyset$, otherwise $x$ would be a cone point of $\Delta$.
Moreover, we have $\delta$ is finite because $\Delta$ is strongly connected.
Furthermore, we have $\delta \geq 2$; otherwise, $\lk_\Delta(\tau)$ would contain more than a single point.

Take the face $\alpha = \sigma_{\delta} \cap \sigma_{\delta-1} \cap \sigma_{\delta-2}$.  
Note that $\dim \alpha = d-2$.  There must exist three vertices $u,v,w$ of $\Delta$ such that:
\[
\sigma_{\delta} = \alpha \cup \{v\} \cup \{w\}, \quad  
\sigma_{\delta-1} = \alpha \cup \{w\} \cup \{x\}, \quad \text{and} \quad  
\sigma_{\delta-2} = \alpha \cup \{x\} \cup \{u\}.
\]

Then, $\lk_\Delta(\alpha)$ contains a path $\{\{v, w\}, \{w, x\}, \{x, u\}\}$.  
Since $\lk_\Delta(\alpha)$ is a 1-dimensional Gorenstein simplicial complex that includes $\{\{v, w\}, \{w, x\}, \{x, u\}\}$, it must form a cycle. Consequently, there exists a path $\{u, y_1, \dots, y_k, v\}$ in $\lk_\Delta(\alpha)$ connecting $v$ and $u$ such that $y_i\notin\{x,w\}$ for all $i=1,\ldots ,k$. Then $\xi=\alpha\cup \{u,y_1\}$ is a facet of $\Delta$ belonging to $X$, and $\sigma_0,\sigma_1,\ldots ,\sigma_{\delta-2},\sigma_{\delta-1}=\xi$ is such that $\sigma_{\delta-1} \in X, \; \sigma_i \in \mathcal{F}_d(\Delta)  \text{\; and \;} |\sigma_i \cap \sigma_{i+1}| = d - 1 \text{ for any } i = 0, \dots, \delta-2$, contradicting the minimality of $\delta$.
\end{proof}

\begin{corollary}\label{c:puncGisHomologyMani}
Let $\Delta$ be a connected simplicial complex with $\dim(\Delta) \ge 2$.
Assume that $\kk[\Delta]$ is not Gorenstein.
If $\kk[\Delta]$ is Gorenstein on the punctured spectrum, then $\Delta$ is a $\kk$-homology manifold.
\end{corollary}
\begin{proof}
By Proposition \ref{prop:main1}
it is enough to show that there is no cone point of $\Delta$.
Assume that there is a cone point $v$ of $\Delta$.
Since $\kk[\Delta]$ is Gorenstein on the punctured spectrum, it follows that $\kk[\lk_\Delta(v)]$ is Gorenstein. Consequently, $R \cong \kk[\lk_\Delta(v)][v]$ is Gorenstein. This yields a contradiction.
\end{proof}
\section{The canonical trace of Stanley--Reisner rings of homology manifolds}\label{sect3}
In this section, we prove Propoition \ref{prop:main2} by using Gr\"abe's structure theorem of canonical modules  and then complete the proofs of our two main theorems, Theorem \ref{main1} and Theorem \ref{main2}.
First we recall Gr\"abe's structure theorem for canonical modules of Stanley--Reisner rings.

Let $S=\kk[x_1,\dots,x_n]$ be the polynomial ring over a field $\kk$.
Let $\ee_i \in \ZZ^n$ be the $i$th unit vector of $\ZZ^n$.
We consider the $\ZZ^n$-grading of $S$ defined by $\deg x_i = \ee_i$.
For a $\ZZ^n$-graded $S$-module $M$ and for $\aaa=(a_1,\dots,a_n) \in \ZZ^n$,
let $M_\aaa$ be the graded component of $M$ of degree $\aaa$
and let $s(\aaa)=\{i \in [n]: a_i \ne 0\}$.

Fix a simplicial complex $\Delta$ of dimension $d \geq 1$. Given a face $\sigma\in\Delta$, one can check (e.g. see \cite{grabe1984canonical}) that:
\[\widetilde{H}_{i-|\sigma|}(\lk_{\Delta}(\sigma);\kk)\cong H_i(\Delta,\cost_\Delta(\sigma);\kk),\]
where on the right hand side we have the relative homology of the pair $(\Delta,\cost_\Delta(\sigma))$. Gr\"abe tells us that there is an isomorphism of $\kk$-vector spaces:
\begin{align*}
\omega_{K[\Delta]} \cong \bigoplus_{\aaa \in \NN^n,\ s(\aaa) \in \Delta} H_{d}\big(\Delta,\cost_\Delta\big(s(\aaa) \big) \big),
\end{align*}
He also tells us which $\kk[\Delta]$-module structure put on the right hand side so that the above isomorphism becomes an isomorphism of $\kk[\Delta]$-modules: if $\tau\subseteq \sigma\in\Delta$ then $\cost_\Delta(\tau)\subseteq \cost_\Delta(\sigma)$, so there is a natural map between relative homologies $\imath^* : H_{d}(\Delta,\cost_\Delta(\tau)) \to H_{d}(\Delta,\cost_\Delta(\sigma))$.



\begin{theorem}[{Gr\"abe \cite{grabe1984canonical}}]\label{t:Grabe}
There exists a following isomorphism as $\ZZ^n$-graded $\kk[\Delta]$-module.
\begin{align}
\label{2-1}
\omega_{K[\Delta]} \cong \bigoplus_{\aaa \in \NN^n,\ s(\aaa) \in \Delta} H_{d}\big(\Delta,\cost_\Delta\big(s(\aaa) \big) \big),
\end{align}
where the $\aaa$th graded component of the right-hand side is $H_{d}(\Delta,\cost_\Delta(s(\aaa)))$.
Moreover, the multiplication structure of the right-hand side determined as follows.
The multiplication of $x_l$ from the $\aaa$th component of the right-hand side of \eqref{2-1} to its $(\aaa+\ee_l)$th component is the following map
\begin{align*} 
\small
\begin{cases}
\mbox{$0$-map}, & \mbox{if }s(\aaa +\ee_l) \not \in \Delta,\\
\mbox{identity map}, & \mbox{if } l \in s(\aaa) \in \Delta,\\
\imath^* : H_{d}(\Delta,\cost_\Delta(s(\aaa ))) \to H_{d}(\Delta,\cost_\Delta(s(\aaa+\ee_l))),&
\mbox{otherwise}.
\end{cases}
\end{align*}
\end{theorem}

In \cite{grabe1984canonical}, Gr\"abe describes an embedding of the canonical module into $R=\kk[\Delta]$:
first of all, $H_d(\Delta, \cost_\Delta \sigma)$ can be embedded into $R$ by
\[
\overline{\sum a_\tau 1_{\tau}} \in H_d(\Delta, \text{cost}_\Delta \sigma)
\longmapsto \sum a_\tau x_\tau x_\sigma \in R.
\]
Let $\kk(\Delta)$ be the ideal of $R$ generated by the images of all $H_d(\Delta, \cost_\Delta \sigma)$, where
$\sigma \in \Delta$.
Then $\omega_R$ is isomorphic to $\kk(\Delta)$ as a $\ZZ$-graded $R$-module up to degree shift by \cite[Theorem 5]{grabe1984canonical}.
Note that $\kk(\Delta)$ contains a nonzero divisor because $R$ is generically Gorenstein~(see \cite[Chapter I,~Theorem 12.9]{stanley2007combinatorics}).

We will need the following to prove various results in this section:
\begin{lemma}\label{l:folklore}
Let $\Delta$ be a $d$-dimensional normal pseudomanifold. We have:
\begin{enumerate}
\item If $\sigma\in\Delta$, $\dim_\kk H_d(\Delta ,\cost_{\Delta}(\sigma);\kk)\leq 1$.
\item If $\tau\subseteq \sigma\in\Delta$ and $H_d(\Delta ,\cost_{\Delta}(\tau);\kk)\neq 0$, then $\iota^*:H_d(\Delta ,\cost_{\Delta}(\tau);\kk)\to H_d(\Delta ,\cost_{\Delta}(\sigma);\kk)$ is an isomorphism. In particular $H_d(\Delta ,\cost_{\Delta}(\sigma);\kk)\cong H_d(\Delta ,\cost_{\Delta}(\tau);\kk)\cong \kk$.
\end{enumerate}
\end{lemma}
\begin{proof}
Point (1) follows at once from Lemma \ref{l:orientability}: Indeed, since $\Delta$ is a normal pseudomanifold, $\lk_{\Delta}(\sigma)$ is a pseudomanifold (of dimension $d-|\sigma|$), so $H_d(\Delta,\cost_\Delta(\sigma);\kk)\cong \widetilde{H}_{d-|\sigma|}(\lk_{\Delta}(\sigma);\kk)$ is either 0 or $\kk$.

As for (2), Gr\"abe proved that $\iota^*$ is injective for all $\tau\subseteq \sigma\in\Delta$ whenever $\Delta$ is a connected quasi-manifold, i.e. a normal simplicial complex such that any face of dimension $\dim\Delta-1$ is contained in at most two facets, see Hauptlemma 3.2 in \cite{grabe1984quasimanifold}. Since a normal pseudomanifold is a connected quasi-manifold, in the situation (2), the map $\iota^*$ is an injective map of $\kk$-vector spaces, so we conclude using (1).
\end{proof}

Given a simplicial complex $\Delta$ on $n$ vertices of dimension $d\geq 1$, call $\min\subseteq \NN^n$ the set of minimal vectors (componentwise) $\aaa\in\NN^n$ such that $[\omega_{\kk[\Delta]}]_{\aaa}\neq 0$, and define $\omega_{\kk[\Delta]}^{\min}$ be the $\kk[\Delta]$-submodule of $\omega_{\kk[\Delta]}$ generated by  $[\omega_{\kk[\Delta]}]_{\aaa}$ where $\aaa\in\min$.  
\begin{proposition}\label{p:normalpseudo=>level}
Let $\Delta$ be a $d$-dimensional normal pseudomanifold and $R=\kk[\Delta]$. Then $\omega_{R}^{\min}=\omega_{R}$. In particular, if $\kk$ has characteristic 2, $\omega_{R}=R[\omega_{R}]_0\cong R$.

Considering the $\ZZ$-grading on $\omega_{R}$, if $\kk$ has characteristic $\neq 2$ and $\Delta$ is a $\kk$-homology manifold, 
\begin{enumerate}
\item $\Delta$ is orientable if and only if $\omega_{R}=R[\omega_{R}]_0\cong R$.
\item $\Delta$ is not orientable if and only if $\omega_{R}=R[\omega_{R}]_1$.
\end{enumerate}
\end{proposition}
\begin{proof}
This follows at once by the structure theorem of Gr\"abe, Lemma \ref{l:orientability} and Lemma \ref{l:folklore}.
For the second part, note that if $\Delta$ is non-orientable and the field $\kk$ has characteristic different from 2, we have $\min = \{\eee_i : i = 1, \ldots, n\}$.
\end{proof}


\begin{proposition}\label{p:nearlyG=G}
Let $\Delta$ be a non-$\kk$-orientable normal pseudomanifold.
Assume that $\kk[\Delta]$ is Cohen--Macaulay.
Then we have $\tr(\omega_{\kk[\Delta]})\subseteq \MM_{\kk[\Delta]}^2$.
\end{proposition}
\begin{proof}
Name $R:=\kk[\Delta]$, and for all $\aaa\in\NN^n$ such that $[\omega_R]_{\aaa}\neq 0$ choose a nonzero vector $1_{\aaa}\in [\omega_R]_{\aaa}$. By Gr\"abe structure theorem and Lemma \ref{l:folklore}, $1_{\aaa}$ is a basis of the $\kk$-vector space $[\omega_R]_{\aaa}$, and if $s(\aaa+\eee_i)\in\Delta$ we have $x_i\cdot 1_{\aaa}=r\cdot 1_{\aaa+\eee_i}$ for some $0\neq r\in\kk$.

Suppose that $\tr(\omega_R)\nsubseteq \MM_R^2$.
Then there exists an element $0 \neq x \in \tr(\omega_R) \cap R_1$.
Since $[\omega_R]_0=0$ (by the non-orientability assumption) and $\Hom_R(\omega_R,R)$ is an $\ZZ^n$-graded $R$-module,
there exist $\aaa\in \NN^n$ and a $\ZZ^n$-graded $R$-homomorphism $\phi: \omega_R \rightarrow R$
such that $\phi(1_{\aaa})=x_i$ for some $i\in[n]$.

Since the annihilator of $x_i$ must be contained in that of $1_{\aaa}$, by the structure theorem of Gr\"abe we get $s(\aaa)=\{i\}$. Moreover, if $\aaa\neq \eee_i$ then $\phi(1_{\bbb})$
(where $\bbb=\aaa-\eee_i$) would be an element of $R$ that multiplied by $x_i$ gives $x_i$. Hence $\phi(1_{\bbb})=1$, which is a contradiction again because the annihilator of $1$ should be contained in that of $1_{\bbb}$. So, summarizing, we can assume that there exists a $\ZZ^n$-graded homomorphism $\phi: \omega_{R} \rightarrow R$
such that $\phi(1_{\eee_i})=x_i$ for some $i\in[n]$.

Take any $x_j \in \lk_{\Delta}(x_i)$,
then there exists a $c \in \kk$ such that $\phi(1_{\eee_j})=cx_j$.
For what said above we have
$x_j 1_{\eee_i}=r x_i 1_{\eee_j}$ for some $0 \neq r \in \kk$.
Thus we have $x_ix_j
=\phi(x_j 1_{\eee_i})=\phi(rx_i 1_{\eee_j})=rx_i \phi(1_{\eee_j})=rcx_ix_j$.
Then we obtain $c\neq 0$ ($x_ix_j \neq 0$ because $x_j \in \lk_{\Delta}(x_i)$).

Since $\Delta$, being normal, is in particular connected,
we have that $\phi(1_{\eee_k})$ is a nonzero scalar multiple of $x_k$ for any $1 \le k \le n$.
Thus $\tr(\omega_R)=\phi(\omega_R)(=\MM_R)$. This contradicts \cite[Theorem 2.1]{gasanova2022rings}.
\end{proof}

From the above, we obtain the following result, which provides an affirmative answer to \cite[Question 4.6]{miyashita2024levelness} and was the main motivation to write this paper. 

\begin{corollary}\label{c:nearlyG=G}
Every nearly Gorenstein Stanley--Reisner ring of dimension greater than two is Gorenstein.
\end{corollary}
\begin{proof}
It follows immediately from Corollary \ref{c:puncGisHomologyMani} and Proposition \ref{p:nearlyG=G}.
Alternatively, it can also be derived from Corollary \ref{c:puncGisHomologyMani}, Proposition \ref{p:normalpseudo=>level} and Proposition \ref{p:Wednesday}.
\end{proof}

\begin{remark}
As mentioned in its proof, Corollary \ref{c:nearlyG=G} can also be regarded as a special case of the following Proposition \ref{p:Wednesday} for a more general class of rings. The proof of this proposition in the case of affine semigroup rings (essentially a standard graded domain) is provided in \cite[Proposition 2.3]{miyashita2023nearly}. Below, we present a proof for generically Gorenstein rings, which are not necessarily integral domains.
\end{remark}

\begin{proposition}\label{p:Wednesday}
Let $R$ be a Cohen--Macaulay generically Gorenstein standard graded ring. Suppose that $[\omega_R]_{-a_R}$ contains an $R$-regular element~(e.g., a level ring of positive dimension)
and $\dim_\kk \left([\omega_R]_{-a_R} \right) \ge \dim_\kk(R_1)$.
Then $R$ is Gorenstein if and only if $[\tr(\omega_R)]_1$ contains a nonzero divisor.
\end{proposition}
\begin{proof}
We may assume that $R_0$ is infinite.
If $R$ is Gorenstein, then $[\tr(\omega_R)]_1=R_1 \neq 0$ contains a nonzero divisor.
Assume that $R$ is not Gorenstein and there exists a homogeneous element $z \in [\tr(\omega_R)]_1$.
Put $n=\dim_\kk([\omega_R]_{-a_R})$ and $a=-a_R$.
Then by \cite[Proposition 5.15~(2)]{miyashita2024linear},
	the canonical module $\omega_R$ can be seen as a homogeneous fractional ideal $J$ of $R$ minimally generated by $R$-regular elements $f_1,\cdots,f_m$, where $m \ge n$.
	We may assume that $f_1,\cdots,f_n \in J_a$.
    Note that $\tr(\omega_R)=J \cdot J^{-1}$ by
    \cite[Lemma 2.6]{miyashita2024linear}.
    Since $z \in [\tr(\omega_R)]_1$, we can write $z=\sum_{i=1}^m{f_i g_i}$, where $g_1,\cdots,g_m$ are homogeneous elements of $J^{-1}$. Moreover, we have $g_i=0$ for any $n < i \le m$ because $f_1,\cdots,f_n \in J_a$ are $R$-regular. Therefore the vector space $V=[J^{-1}]_{1-a}$ is equal to $\frac{1}{b} \langle x_1,\cdots,x_s \rangle$ where $b$ is a nonzero divisor of $R$, $s$ is natural number and $x_1,\cdots,x_s$ homogeneous elements of $R$ of the same degree, $1-a+\deg(b)$.
	
Then the vector space $V'=\langle x_1,\cdots,x_s \rangle$ contains a nonzero divisor. Otherwise, $V'$ is contained in some associated prime $\pp$ of $R$,
and so $b[\tr(\omega_R)]_1 R\subseteq J V' \subseteq \pp$, thus $[\tr(\omega_R)]_1 R \subseteq \pp$. This yields a contradiction.
Therefore, there exists an $R$-regular element $\frac{y}{b} \in V$, so that $\frac{y}{b} \cdot J$ must contain $n$ linearly independent linear forms of $R$, hence $\frac{y}{b} \cdot J\supset \MM_R$.
So we have
$n = \dim_\kk(R_1)$ and, since $R$ is not Gorenstein, $\tr(\omega_R)=\frac{y}{b} \cdot J$.
Thus, \( R \) is of Teter type, a contradiction to \cite[Theorem 2.1]{gasanova2022rings}.
\end{proof}

\begin{remark}
If $\kk$ has characteristic 2, we can weaken the hypothesis in Corollary \ref{c:nearlyG=G}: a Cohen--Macaulay Stanley--Reisner ring of dimension greater than two is Gorenstein as soon as it is Gorenstein on the punctured spectrum.
If $\kk$ has characteristic not equal to 2, there exist Cohen--Macaulay Stanley--Reisner rings that are Gorenstein on the punctured spectrum but not nearly Gorenstein. For example, the Stanley--Reisner ring $\kk[\Delta]$ of the minimal triangulation $\Delta$ of the real projective plane is such a case (see \cite[Example 4.5]{miyashita2024levelness}). Our next purpose is to compute $\tr(\omega_{\kk[\Delta]})$ in these situations.
\end{remark}

\begin{proposition}\label{prop:main2}
Let $\Delta$ be a non-$\kk$-orientable $\kk$-homology manifold. Then we have $\tr(\omega_{\kk[\Delta]}) \supseteq \MM_{\kk[\Delta]}^2$.
\end{proposition}
\begin{proof}
Let $R=\kk[\Delta]$.
and fix $i\in[n]$. 
By Lemma~\ref{l:folklore}~(2), note that $x_i:[\omega_R]_{\eee_j} \rightarrow [\omega_R]_{\eee_i+\eee_j}$ is a surjection for any $j \in \lk_\Delta(x_i)$.
Thus we can choose a set of $\kk$-basis $\{ 1_{\{k\}} : k \text{\;is a vertex of\;} \Delta \}$ for $[\omega_R]_1$ such that $x_i \cdot 1_{\{j\}} = x_j \cdot 1_{\{i\}}$ holds for any $j \in \lk_\Delta(x_i)$.
Let $\iota: \omega_R \hookrightarrow R$ be the natural injection defined in \cite[Theorem 5]{grabe1984canonical}, such that $\iota(\omega_R) = \kk(\Delta)$.
Notice that $\kk(\Delta)=R(\iota(1_{ \{1\} }),\ldots,\iota(1_{ \{n\} }))$.
Consider 
\[f_i=\iota(1_{\{i\}})+\sum_{k \text{\;is a vertex of\;} \Delta \text{\;such that\;} k \notin \st(i) }x_k^{\dim(\Delta)+1}.\]
Then $f_i$ is a homogeneous nonzero divisor of $R$
because $x^\aaa f_i \neq 0$ for any $\aaa \in \NN^n$ such that $s(\aaa) \in \FF(\Delta)$.
Notice that $x_i \cdot \iota(1_{\{j\}}) = x_j \cdot \iota(1_{\{i\}})$ holds for any $j \in \lk_\Delta(x_i)$.
Therefore, for any $j \in [n]$, we have

\[
\frac{x_i^2}{f_i} \cdot \iota(1_{ \{j\} }) = 
\begin{cases}
    x_ix_j, & \text{if \;} j \text{\;is a vertex of $\st(i)$}, \\
    0, & \text{otherwise}.
\end{cases} \qquad
\]
In particular $\displaystyle \frac{x_i^2}{f_i} \in \kk(\Delta)^{-1}$. Since $\tr(\omega_R)=\kk(\Delta) \cdot \kk(\Delta)^{-1}$ by \cite[Lemma 1.1]{herzog2019trace}, we obtain 
$$\tr(\omega_R) \supseteq (x_ix_j : j \text{\;is a vertex of $\st(i)$}).$$
Since $i \in [n]$ is arbitrary, we have $\tr(\omega_R) \supseteq 
\sum_{i \in [n]} (x_ix_j : j \text{\;is a vertex of $\st(i)$})
=\MM_R^2$.
\end{proof}

\begin{corollary}\label{cor:main2}
Let $\Delta$ be a non-oriantable $\kk$-homology manifold on $n$ vertices.
Assume that $\kk[\Delta]$ is Cohen--Macaulay.
Then $\kk[\Delta]$ is level of Cohen--Macaulay type $n$. Moreover, we have $\tr(\omega_R)=\MM_R^2$.
\end{corollary}
\begin{proof}
It follows from
Proposition \ref{p:normalpseudo=>level},
Proposition \ref{p:nearlyG=G}
and Proposition \ref{prop:main2}.
\end{proof}

\begin{remark}
In Proposition \ref{p:nearlyG=G}, we assume that $R$ is Cohen--Macaulay; however, even in the non-Cohen--Macaulay case, it follows from the proof above that $\omega_R$ is generated by $[\omega_R]_1$ as an $R$-module and that $\tr(\omega_R) \supseteq \MM_R^2$.
\end{remark}

We are now in a position to establish Theorem \ref{main1} and Theorem \ref{main2}.
\begin{proof}[Proof of Theorem \ref{main1}]
First we show (X).
It is evident from \cite[Lemma 2.1]{herzog2019trace} that the Cohen--Macaulay ring $R=\kk[\Delta]$ is Gorenstein on the punctured spectrum when $\tr(\omega_R) = \MM_R^i$ for some $i \in \{0,1,2\}$.
Assume that $R$ is Gorenstein on the punctured spectrum.
When $\dim(\Delta) \leq 1$, $\Delta$ is nearly Gorenstein if and only if it is Gorenstein on the punctured spectrum by \cite[Theorem 4.3~(a) and (b)]{miyashita2024levelness}, so if $\dim \Delta\leq 1$ and is $R$ Gorenstein on the punctured spectrum, either $\tr(\omega_R)=R$ or $\tr(\omega_R)=\MM_R$. If $\dim(\Delta) \ge 2$,
then either $R$ is Gorenstein or $\Delta$ is a $\kk$-homology manifold by Corollary \ref{c:puncGisHomologyMani}.
Note that, according to Proposition \ref{p:normalpseudo=>level}, $\kk[\Delta]$ is Gorenstein if $\Delta$ is $\kk$-orientable.
Thus we have either $\tr(\omega_R)=R$ or $\tr(\omega_R)=\MM_R^2$ by Corollary \ref{cor:main2}.
Therefore, in each case, it follows that $\tr(\omega_R) = \MM_R^i$ for some $i \in \{0,1,2\}$.

Next we show (Y).
When $\dim(\Delta) \ge 2$, $R$ is nearly Gorenstein if and only if $R$ is Gorenstein by Corollary \ref{c:nearlyG=G}.
Then we have $\dim(\Delta) \le 1$, in which case the statement follows from \cite[Theorem 4.3~(a) and (b)]{miyashita2024levelness}.
Lastly, we show (Z).
(2) $\Rightarrow$ (1) follows from Corollary \ref{cor:main2}.
Assume that $\tr(\omega_R)=\MM_R^2$.
Since $R$ is Gorenstein on the punctured spectrum but not nearly Gorenstein, we conclude that $\dim(\Delta) \ge 2$. Otherwise, by \cite[Theorem 4.3~(a) and (b)]{miyashita2024levelness}, $R$ would be nearly Gorenstein because it is Gorenstein on the punctured spectrum, leading to a contradiction.
Therefore, $\Delta$ is a non-$\kk$-orientable $\kk$-homology manifold by Corollary \ref{c:puncGisHomologyMani} and Proposition \ref{p:normalpseudo=>level}.
\end{proof}


\begin{proof}[Proof of Theorem \ref{main2}]
Concerning $(X)$, if $\dim(\Delta)\leq 1$ it follows by \cite[Theorem 4.3]{miyashita2024levelness}, when $\dim(\Delta)\geq 2$ by Corollary \ref{c:puncGisHomologyMani} and Proposition \ref{p:normalpseudo=>level}.
Lastly, we show $(Y)$.
(3) $\Leftrightarrow (4)$ follows from Theorem \ref{main1}.
Notice that every Stanley--Reisner ring of $0$-dimensional simplicial complex is almost Gorenstein.
In fact, since it has minimal multiplicity and is nearly Gorenstein by \cite[Theorem 4.3(a)]{miyashita2024levelness}, this follows from \cite[Theorem 6.6]{herzog2019trace} (or alternatively from \cite[Corollary 4.8(b)]{miyashita2024comparing}).
Thus (4) $\Rightarrow$ (2) follows from \cite[Theorem 4.3~(c)]{miyashita2024levelness} (see also \cite[Proposition~3.8]{matsuoka2016uniformly}).
(2) $\Rightarrow$ (1) is clear.
We show that 
(1) $\Rightarrow$ (4).
It is enough to show that $\dim(\Delta)<2$ by Theorem \ref{main1}.
Suppose that $\dim(\Delta) \ge 2$.
Notice that $R$ is level by Theorem \ref{main2}~(X).
Thus the socle degree $s(R)$ of $R$ should be equal to 1 by \cite[Theorem 5.1]{miyashita2024comparing}.
At this point, the codimension of $R$ is equal to its Cohen--Macaulay type. Therefore, by Corollary \ref{c:puncGisHomologyMani} and Proposition \ref{p:normalpseudo=>level}, the codimension of $R$ equals its embedding dimension. Consequently, we obtain that $\dim(R) = 0$, which contradicts the fact that $\dim(R) \ge 1$.
\end{proof}

\section*{Acknowledgments}
The first author gratefully acknowledges his supervisor, Akihiro Higashitani, for providing the opportunity to visit the University of Genoa.
The second author was is supported by PRIN~2020355B8Y ``Squarefree Gr\"obner degenerations, special varieties and related topics,'' by MIUR Excellence Department Project awarded to the Dept.~of Mathematics, Univ.~of Genova, CUP D33C23001110001.


\end{document}